\documentclass{amsart}
\usepackage{times, txfonts, verbatim, bbm, mathrsfs, enumerate, array, %
longtable,  %
amssymb}%
\usepackage{pstricks, pst-node, pst-tree}
\usepackage[initials]{amsrefs}

\title {Free cubic implication algebras}	% insert your name here
\author{Colin G.Bailey}
\address{School of Mathematics, Statistics \& Operations Research\\
Victoria University of Wellington\\
PO Box 600\\
Wellington\\
NEW ZEALAND
}
\email{Colin.Bailey@vuw.ac.nz}
\author{Joseph S.Oliveira}
\address{
Pacific Northwest National Laboratories\\
Richland\\
U.S.A.}
\email{Joseph.Oliveira@pnl.gov}
\dedicatory{To the memory of our mentor and friend \\
Gian-Carlo Rota}
\date{\today}
\let\Cal\mathcal
\let\rsf\mathscr

\newcommand{\card}[1]{\left| #1\right|}
\def\one{\mathbf1}
\def\caret{\hat{\ }}
\newcommand{\Item}[1]{\Tr[ref=l]{%
\psframebox[linestyle=none]{#1}}}

\def\one{\mathbf1}
\providecommand{\meet}{\mathbin{\wedge}}
\providecommand{\join}{\mathbin{\vee}}
\newcommand{\comp}[1]{\overline{#1}}
\ifx\upharpoonright\undefined
     \def\restrict{\hbox{\rm\kern0.166em\accent"12\kern-0.536em$\vert$\kern0.3em}}%
  \else
     \def\restrict{\upharpoonright}%
  \fi
\makeatletter
\def\twoSet#1#2{\left\{%
\vphantom{#2}#1\thinspace\right|\nolinebreak[3]\left.%
  #2%
  \vphantom{#1}%
  \right\}%
}
\def\oneSet#1{\left\lbrace#1\right\rbrace}

\newif\if@nstr
\def\setstrfalse{\let\if@nstr=\iffalse}
\def\setstrtrue{\let\if@nstr=\iftrue}
\def\@nstr #1#2{
\def\@@nstr ##1#1##2##3\@@nstr{\ifx
\@nstr ##2\setstrfalse \else \setstrtrue \fi }
\@@nstr #2#1\@nstr \@@nstr}
\def\@separate#1|#2@{\setFront{#1}\setBack{#2}}
\def\lb#1\rb{\@nstr|{#1} \if@nstr \@separate#1 @ \twoSet{\@setFront}{\@setBack}%
\else \@separate |{#1 }@ \oneSet{\@setBack}\fi%
}
\def\setFront#1{\def\@setFront{#1}}
\def\setBack#1{\def\@setBack{#1}}
\def\Set#1{\lb{#1}\rb}
\def\oneBrk#1{\left\langle#1\right\rangle}
\def\twoBrk#1#2{\left\langle%
\vphantom{#2}#1\thinspace\right|\nolinebreak[3]\left.%
  #2%
  \vphantom{#1}%
  \right\rangle%
}
\def\brk<#1>{\@nstr|{#1} \if@nstr \@separate#1 @ \twoBrk{\@setFront}{\@setBack}%
\else \@separate |{#1 }@ \oneBrk{\@setBack}\fi%
}
\def\lemref#1{\normalfont{lemma}~\ref{#1}}

\theoremstyle{plain}
\newtheorem{thm}{Theorem}[section]
\newtheorem{lem}[thm]{Lemma}
\newtheorem{cor}[thm]{Corollary}
\newtheorem{prop}[thm]{Proposition}
\newtheorem{defn}[thm]{Definition}

\theoremstyle{remark}
{}
{\newtheorem{example}{Example}[section]}
{}
{}

\@ifpackageloaded{bbm}{

\newcommand{\N}{{\mathbbm{N}}}

}{

\newcommand{\N}{{\mathbb{N}}}

}
\makeatother

\begin{document}
\begin{abstract}
	We construct free cubic implication algebras with finitely many generators, 
	and determine the size of these algebras. 
\end{abstract}
\maketitle
\section{Introduction}
In \cite{MR:cubes} Metropolis and Rota introduced a new way of looking 
at the face lattice of an $n$-cube based on its symmetries. 
Subsequent work has lead to a  purely  equational representation of 
these lattices -- the varieties of MR and cubic implication algebras. 
\cite{BO:UniMR} describes the variety of Metropolis-Rota implication 
algebras (MR algebras). \cite{BO:eq} gives 
an equational description of cubic implication algebras and implicitly proves 
that the face lattices of $n$-cubes generate the variety. 
Therefore free 
cubic implication algebras must exist. In this paper we give an explicit construction 
for the free cubic implication algebra on $m$ generators and determine its size. 

The argument comes in several parts.  First we produce a candidate for 
the free algebra on $k+1$ generators by looking at embeddings into 
interval algebras and choosing a minimal one. Thus the free algebra 
is embedded into a known cubic implication algebra. We compute the 
size of this cubic implication algebra. From \cite{BO:eq} we know that every 
finite cubic implication algebra is a finite union of interval algebras and we 
know that the overlaps are also interval algebras. The size of an 
interval algebra is easy to determine so we can use an 
inclusion-exclusion argument to determine the size of the cubic 
algebra. 

The next part is to show that our candidate for the free algebra is 
generated by the images of the generators and so the embedding is onto. 
Again we use the facts that our cubic implication algebra is a finite union of 
interval algebras and each interval algebra is the set of 
$\Delta$-images of a Boolean algebra to reduce the problem to showing 
that certain atoms in a well-chosen Boolean algebra are generated.

We start by recalling some basic definitions and facts about cubic 
algebras -- the reader is referred to \cites{BO:eq} for more 
details. 
\begin{defn}
	A \emph{cubic implication algebra} is a join semi-lattice with one and a binary 
	operation $\Delta$ satisfying the following axioms:
	\begin{enumerate}[a.]
		\item  if $x\le y$ then $\Delta(y, x)\join x = y$;
		
		\item  if $x\le y\le z$ then $\Delta(z, \Delta(y, x))=\Delta(\Delta(z, 
		y), \Delta(z, x))$;
		
		\item  if $x\le y$ then $\Delta(y, \Delta(y, x))=x$;
		
		\item  if $x\le y\le z$ then $\Delta(z, x)\le \Delta(z, y)$;
		
		\item[] Let $xy=\Delta(\one, \Delta(x\join y, y))\join y$ for any $x$, $y$ 
		in $\mathcal L$. Then:
		
		\item  $(xy)y=x\join y$;
		
		\item  $x(yz)=y(xz)$;
	\end{enumerate}
\end{defn} 

\begin{defn}
	An \emph{MR implication algebra} is a cubic implication algebra satisfying the MR-axiom:\\
	if $a, b<x$ then 
	\begin{gather*}
		\Delta(x, a)\join b<x\text{ iff }a\meet b\text{ does not exist.}
	\end{gather*}
\end{defn}

\begin{example}\label{def:intAlg}
    Let $B$ be a Boolean algebra, then the interval algebra of $B$ is
    $$
    \rsf I(B)=\Set{[a, b] | a\le b \text{ in }B}
    $$
    ordered by inclusion. The operations are defined by
    \begin{align*}
        \one &=[0, 1]\\
        [a, b]\join[c, d]&=[a\meet c, b\join d]\\
        \Delta([a, b], [c, d])&=[a\join(b\meet\comp d), b\meet(a\join\comp c)].
    \end{align*}
    It is straightforward to show $\rsf I(B)$ is an MR-algebra. Additional details
may be found in \cite{BO:eq}.
\end{example}

\begin{example}\label{def:signedSet}
    Let $X$ be any set. The \emph{signed set algebra of $X$} is the 
    set
    $$
    \rsf S(X)=\Set{\brk<A, B> | A, B\subseteq X\text{ and }A\cap 
    B=\emptyset}. 
    $$
    The operations are 
    \begin{align*}
        \one & =\brk<\emptyset, \emptyset>  \\
	\brk<A, B>\join\brk<C, D> & =\brk<A\cap C, B\cap D>  \\
        \Delta(\brk<A, B>, \brk<C, D>) & =\brk<A\cup C\setminus B, 
        B\cup D\setminus A>. 
    \end{align*}
    
    $\rsf S(X)$ is isomorphic to $\rsf I(\wp(X))$ by $\brk<A, B>\mapsto [A, 
X\setminus B]$. All finite MR-algebras are isomorphic to some signed 
set algebra (and hence to some interval algebra).
\end{example}

As part of the representation theory in  \cite{BO:eq} we had the 
following definitions and lemma:
\begin{defn}
	Let $\mathcal L$ be a cubic implication algebra and 
	$a\in\mathcal L$. Then the \emph{localization of $\mathcal L$ 
	at $a$} is the set
	$$
	\mathcal L_{a}=\Set{\Delta(y, x) | a\le x\le y}.
	$$
\end{defn}

Associated with localization is the binary relation $\preccurlyeq$ 
we can define by
$$
a\preccurlyeq b\text{ iff }b\in\mathcal L_{a}
$$
or by the equivalent internal definitions:
$$
a\preccurlyeq b \text{ iff }a\le \Delta(a\join b, b)
$$
$$
a\preccurlyeq b \text{ iff }b=(a\join b)\meet(\Delta(\one, a)\join b).
$$
In \cite{BO:eq} we establish the equivalence of these three 
definitions and make great use of this relation in getting a 
representation theorem for cubic algebras. The next lemma is the 
 part of that representation theory that we need in order to 
 understand the free algebra construction that follows.

\begin{lem}\label{lem:kl}\label{lem:onto}
	Let $\mathcal L$ be any cubic implication algebra. Then
	$\mathcal L_{a}$ is an atomic MR-algebra, and hence 
	isomorphic to an interval algebra.
\end{lem}

\section{Free Algebras}
\begin{defn} Let $X$ be a set,  the set of generators. 
\begin{enumerate}[(i)]
	\item  Let $\Cal Fr(X)$ denote the free cubic implication algebra with generators $X$. 
	
	\item  Let $X'=\Set{s_x|x\in X}\cup\Set{t_x|x\in X}$ and let
	$B$ be the Boolean algebra generated by $X'$ with the relations 
	$s_x\le t_x$ for all $x\in X$. Let $x'=[s_x, t_x]\in\rsf I(B)=\Cal L$,  and let
	$\Cal L_{1}(X)=\bigcup_{x\in X}\Cal L_{x'}$. 
	
	\item Let $\mathcal F_{k}$ be the free Boolean algebra with the $2k+2$ 
	generators $\Set{s_{0}, \dots, s_{k}}\cup\Set{t_{0}, \dots, t_{k}}$.
\end{enumerate}	
\end{defn}

It is not the case that $\Cal L_{1}(X)$ is the free cubic implication algebra,  it is too 
large. But it serves as a prototype for discussing the
construction of a free cubic implication algebra. It also allows us to compute an 
upper bound to the size of the free algebra.

\begin{lem}\label{lem:finiteFA}
	$\Cal Fr(\Set{x_{0}, \dots, x_{k}})$ is finite with size at most
	$3^{{2^{2k+2}}}$. 
\end{lem}
\begin{proof}
	This is because $\Cal Fr(\Set{x_{0}, \dots, x_{k}})$ embeds into
	$\rsf I(\mathcal F_{k})$ -- letting $x_{i}=[s_{i}, t_{i}]$. 
\end{proof}

The idea of this proof is crucial -- we embed the free algebra into 
an interval algebra and determine properties of the free algebra from 
the embedding. 

Suppose that $X=\{a_0, a_1, \dots, a_k\}$ is finite. Let $\Cal Fr(X)$ 
embed into $\rsf I(C)=\Cal L$ as an upper segment for some Boolean algebra $C$. 
Let $e(a_i)=[s_i, t_i]=a_i'$ for $0\le i\le k$. Define inductively 
\begin{align*}
	\delta_0&=a_0'\\
	\delta_{i+1}&=\delta_{i}\caret a'_{i+1}\\
	&=\delta_i\meet\Delta(\delta_i\join a'_{i+1}, a'_{i+1}). 	
\end{align*}

% We know that the meet exists, by the MR axiom, as 
% $\delta_i\join a_{i+1}=\delta_i\join a_{i+1}$. 

Then it must be the case that $\Cal Fr(X)$ embeds into $\Cal 
L_{\delta_k}$, since we have that $\delta_k\preccurlyeq a_i'$ for all $i$ and 
the image of
$\Cal Fr(X)$ in $\rsf I(C)$ is the set $\bigcup_{i=0}^k\Cal L_{a_i'}$. 
Thus we may as well assume that $\delta_k$ is a vertex, and furthermore 
that it is $[0, 0]$ as all vertices are interchangeable by a cubic 
isomorphism. Furthermore, we see that if $B^*$ is the subalgebra 
generated by the $s_i$'s and the $t_i$'s, then in fact $\Cal Fr(X)$ 
embeds into $\rsf I(B^*)_{\delta_k}$. 

So now we construct a new candidate for $\Cal Fr(X)$, where $X$ is the 
finite set
$\{a_0, a_1, \dots, a_k\}$. Let $\Cal B_X$ be the Boolean algebra 
generated by
$\{s_0, \dots, s_k\}\cup\{t_0, \dots, t_k\}$ with the relations
\begin{align*}
	s_i&\le t_i\qquad\text{ for all }0\le i\le k\\
	\delta_k&=[0, 0]. 	
\end{align*}

By the above argument, we see that 
$$\card{\Cal Fr(X)}\le\card{\bigcup_{i=0}^k\Cal L_{[s_i, t_i]}}. $$
We will compute the cardinality of the right-hand-side 
and show that the intervals $[s_{i}, t_{i}]$ cubically generate 
${\bigcup_{i=0}^k\Cal L_{[s_i, t_i]}}$ and so 
$\Cal Fr(X)\simeq\bigcup_{i=0}^k\Cal L_{[s_i, t_i]}$.

\section{Getting Better Relations}
The relation $\delta_k=[0, 0]$ is not easy to use,  so we will recast it 
as a series of statements about the $s_i$'s and the $t_i$'s. 

First a fact about interval algebras that we will often make use of 
in the following argument. It is easily verified from the definitions 
above.
\begin{quote}
	If $a=[a, b]$ and $w=[u, v]$ are any two intervals then\\
$$
	a\meet \Delta(a\join w, w)=
	[a\join(b\meet\comp v), b\meet(a\join\comp u)].
$$
\end{quote}
% \begin{proof}
% 	\begin{align*}
% 		w&=[u, v]\\
% 		a\joinw&=[a\meet u, b\join v]\\
% 		\Delta(a\joinw, w)&=
% 		   [(a\meet u)\join ((b\join v)\meet \comp v), (b\join v)\meet ((a\meet 
% 		   u)\join\comp u)]\\
% 		   &=[(a\meet u)\join(b\meet\comp v), (b\join v)\meet(a\join\comp 
% 		   u)]\\
% 		a\meet \Delta(a\joinw, w)&=
% 		   [a\join(a\meet u)\join(b\meet\comp v), b\meet(b\join 
% 		   v)\meet(a\join\comp u)]\\
% 		   &=[a\join(b\meet\comp v), b\meet(a\join\comp u)]. 
% 	\end{align*}
% 	
% 	We recall, here for convenience,  that $a\le b$ implies 
% 	\begin{align*}
% 		b\meet(a\join\comp u)&=(b\meet a)\join(b\meet\comp u)\\
% 		&=a\join(b\meet\comp u). 
% 	\end{align*}
% \end{proof}

Now define inductively a sequence from $\Cal B_X$ as follows:\\
\begin{center}
\begin{tabular}{>{$}r<{$}!{ $=$ }>{$}l<{$}!{\hspace{1cm}}>{$}r<{$}!{ $=$ }>{$}l<{$}}
	\sigma_0 & s_0;&\tau_0 & t_0;\\
	\sigma_{i+1} & \sigma_i\join(\tau_i\meet\comp t_{i+1}); & %
	\tau_{i+1} & \sigma_i\join(\tau_i\meet\comp s_{i+1}). 
\end{tabular}
\end{center}
It is not hard to see that $\sigma_i\le \tau_i$ for all $i$ and that 
$\delta_i=[\sigma_i, \tau_i]$. So our extra condition can now be 
rewritten as $\sigma_k=\tau_k=0$. This is still rather unsatisfactory. 
Instead of using these relations we will produce another set that 
give useful information more directly. We do this by
defining a larger class of relations that are used to show that the 
desired relations capture the ones above and no more.

\begin{defn}
	Let $0\le l\le k$ in $\N$, and $t, \alpha$ in $B_{X}$. Then 
	\begin{align*}
		R_{l, k, i}(t, \alpha):&& t&\le\bigvee_{j=l}^{i}s_{j}\join 
		t_{i+1}\join\alpha\\
		Q_{l, k}(t, \alpha):&& t &\le \bigvee_{j=l}^{k}s_{j}\join\alpha.
	\end{align*}
\end{defn}

We now demonstrate that these are the desired relations.

\begin{lem}
	$$
	\text{If }\tau_{k}\le\alpha \text{ then }t_{0}\le\bigvee_{j=1}^{k}s_{j}\join\alpha.
	$$
\end{lem}
\begin{proof}
	For $k=0$ this just says that $t_{0}\le\alpha$.
	
	If general we have 
	\begin{align*}
		\tau_{k+1}=\sigma_{k}\join(\tau_{k}\meet\comp 
		s_{k+1})\le\alpha&\iff 
		\sigma_{k}\le\alpha\text{ and }\tau_{k}\meet\comp s_{k+1}\le\alpha\\
		&\hphantom{\Leftarrow}\Rightarrow
		\tau_{k}\le s_{k+1}\join \alpha\\
		&\hphantom{\Leftarrow}\Rightarrow t_{0}\le\bigvee_{j=1}^{k}s_{j}\join s_{k+1}\join 
		\alpha\\
		&\hphantom{\Leftarrow}\Rightarrow t_{0}\le\bigvee_{j=1}^{k+1}s_{j}\join\alpha.
	\end{align*}
\end{proof}

\begin{lem}
	\begin{gather*}
		\text{If }\sigma_{k}\le\alpha\text{ then }\quad
			t_{0}  \le \bigvee_{j=1}^{i}s_{j}\join t_{i+1}\join\alpha  \text{ 
			for all }i<k\text{ and }  
			s_{0} \le \alpha. 
	\end{gather*}
\end{lem}
\begin{proof}
	For $k=0$ this just says that $s_{0}\le\alpha$.
	
	In general we have 
	\begin{align*}
		\sigma_{k+1}=\sigma_{k}\join(\tau_{k}\meet\comp t_{k+1})\le\alpha
		&\iff \sigma_{k}\le\alpha\text{ and }\tau_{k}\le 
		t_{k+1}\join\alpha\\
		&\hphantom{\Leftarrow}\Rightarrow
		\begin{array}[t]{r@{\le}ll}
			t_{0} & \displaystyle\bigvee_{j=1}^{i}s_{j}\join t_{i+1}\join\alpha & \text{ 
			for all }i<k\text{; }  \\
			s_{0} & \alpha\text{; and } &\\
			t_{0} & \displaystyle\bigvee_{j=1}^{k}s_{j}\join 
		t_{k+1}\join\alpha
		\end{array}\\
		&\hphantom{\Leftarrow}\Rightarrow
		\begin{array}[t]{r@{\le}ll}
			t_{0} & \displaystyle\bigvee_{j=1}^{i}s_{j}\join t_{i+1}\join\alpha & \text{ 
			for all }i\le k\text{; and}\hbox{\vrule depth 3mm width0pt 
			height0pt} \\
			s_{0} & \alpha. & 
		\end{array}
	\end{align*}
\end{proof}

This shows the necessity of these relations,  now we show they are 
also sufficient.
\begin{lem}
	Suppose that $i<k$ and
	\begin{align*}
		\sigma_{i}&\le\alpha;\\
		R_{i+1, k, j}(\tau_{i}, \alpha)&\text{ for all }i\le j<k;\text{ 
		and }\\
		Q_{i+1, k}(\tau_{i}, \alpha).& \\
		\intertext{ Then }
		\sigma_{i+1}&\le\alpha;\\
		R_{i+2, k, j}(\tau_{i+1}, \alpha)&\text{ for all }i+1\le j<k;\text{ 
		and }\\
		Q_{i+2, k}(\tau_{i+1}, \alpha).&
	\end{align*}
\end{lem}
\begin{proof}
	We have $\tau_{i}\le\bigvee_{p=i+1}^{j}s_{p}\join t_{j+1}\join\alpha$ for all 
	$i\le j<k$. 
	
	Taking 
	$j=i$ we have that $\tau_{i}\meet\comp t_{i+1}\le\alpha$.
	As $\sigma_{i}\le\alpha$ we therefore get
	$\sigma_{i+1}=\sigma_{i}\join(\tau_{i}\meet\comp t_{i+1})\le\alpha$.
	
	For $j>i$ we get $\tau_{i}\meet\comp 
	s_{i+1}\le\bigvee_{p=i+2}^{j}s_{p}\join t_{j+1}\join\alpha$, and so
	$\tau_{i+1}=\sigma_{i}\join(\tau_{i}\meet\comp s_{i+1})\le
	\alpha\join 
	\bigvee_{p=i+2}^{j}s_{p}\join t_{j+1}\join\alpha= \bigvee_{p=i+2}^{j}s_{p}\join t_{j+1}\join\alpha$.
	
	From $Q_{i+1, k}(\tau_{i}, \alpha)$ we have 
	$\tau_{i}\le\bigvee_{p=i+1}^{k}s_{p}\join\alpha$ and so
	$\tau_{i}\meet\comp s_{i+1}\le 
	\tau_{i}\le\bigvee_{p=i+2}^{k}s_{p}\join\alpha$ which gives
	$\tau_{i+1}\le \tau_{i}\le\bigvee_{p=i+2}^{k}s_{p}\join \alpha$.
\end{proof}

\begin{cor}
	Suppose that 
	\begin{align*}
		s_{0}&\le\alpha;\\
		R_{1,  k, j}(t_{0}, \alpha)&\qquad\text{ for all }1\le j<k;\text{ 
		and }\\
		Q_{1, k}(t_{0}, \alpha).& \\
		\intertext{ Then }
		\sigma_{k}&\le\alpha;\text{ and }\\
		\tau_{k}&\le\alpha.
	\end{align*}
\end{cor}
\begin{proof}
	It follows immediately from the lemma,  and noting that $Q_{k+1, 
	k}(\tau_{k}, \alpha)$ is the same as $\tau_{k}\le \alpha$.
\end{proof}

\begin{prop}
	$\sigma_{k}=\tau_{k}=0$ iff 
	\begin{tabular}[t]{>{$\displaystyle}r<{$}@{}>{$\displaystyle}l<{$}}
		s_{0}&\le0;\\
		R_{1 , k, j}(t_{0}, 0)&\text{ for all }1\le j<k;\text{ 
		and }\\
		Q_{1, k}(t_{0}, 0).
	\end{tabular}
\end{prop}
\begin{proof}
	Immediate from the last corollary.
\end{proof}

\begin{defn}
	Let $\mathcal B_{k}$ be the Boolean algebra generated by
	$\Set{s_{0}, \dots, s_{k}}\cup\Set{t_{0}, \dots, t_{k}}$ with the 
	relations\\
	\begin{tabular}[t]{r>{$\displaystyle}r<{$}@{}>{$\displaystyle}l<{$}}
		$Z$: & s_{0}&\le0;\\
		$S_{i}$: & s_{i}&\le t_{i} \text{ for all }i\le k;\\
		$R_{j}$: & R_{1 , k, j}(t_{0}, 0)&\text{ for all }1\le j<k;\text{ 
		and }\\
		$Q_{k}$: & Q_{1, k}(t_{0}, 0).
	\end{tabular}
	
	Let $\mathcal L^{(k)}=\rsf I(\mathcal B_{k})$ and
	$\mathcal L(X)=\bigcup_{i=0}^{k}\mathcal L^{(k)}_{[s_{i}, t_{i}]}$.
\end{defn}

We aim to compute the size of $\mathcal L(X)$ as this provides an 
upper bound to the size of $\Cal Fr(X)$. This starts with a lot 
of atom counting in $\mathcal B_{k}$.

\section{Looking at Atoms}
In this section we aim to see how atoms are produced in $\mathcal 
B_{k}$ as a preliminary to counting them. This is based upon our 
knowledge of atoms in $\mathcal F_{k}$ and the ideal we quotient out 
by to get $\mathcal B_{k}$. 
This ideal is generated by the set
$$
S_{k}=\Set{s_{0}}\cup\Set{s_{i}\meet\comp t_{i} | 0\le i\le k}\cup
\Set{t_{0}\meet\vphantom{\bigwedge}\smash{\bigwedge_{j=1}^{i}}\comp s_{j}\meet\comp t_{i+1} | 0\le 
	i<k}\cup\Set{t_{0}\meet\vphantom{\bigwedge}\smash{\bigwedge_{j=0}^{k}}\comp s_{j}}. 
	\vphantom{\bigwedge_{j=1}^{i}}
$$
The elements of $S_{k}$ come in four different kinds. For ease of 
reference we name them as
\begin{align*}
	s_{0}& \\
	u_{i}&=s_{i}\meet\comp t_{i}&&\text{ for }0\le i\le k\\
	r_{i}&=t_{0}\meet\bigwedge_{j=1}^{i}\comp s_{j}\meet\comp t_{i+1}&&\text{ 
	for }0\le i\le k-1\\
	q_{k}&=t_{0}\meet\bigwedge_{j=0}^{k}\comp s_{j}. 
\end{align*}

There is one new atom that comes from the failure of $Q_{k-1}$. 
\begin{lem}\label{lem:newAtom}
	Let $a_{k}=t_{0}\meet\bigwedge_{i=1}^{k-1}\comp s_{i}$.
	Then $a_{k}$ is an atom in $\mathcal B_{k}$.
\end{lem}
\begin{proof}
	We recall from the usual construction of the free Boolean algebra
	on $\Set{s_{0}, \dots, s_{k}}\cup\Set{t_{0}, \dots, t_{k}}$ that
	every atom has the form
	$\bigwedge_{j=0}^{k}\varepsilon_{j}s_{j}\meet\bigwedge_{j=0}^{k}\delta_{j}t_{j}$
	where  $\varepsilon_{j}$   and $\delta_{j}$ are $\pm1$ and $1a=a$, 
	$-1a=\comp a$.
	
	In $\mathcal B_{k}$ we have $s_{0}=0$ so that $\varepsilon_{0}=-1$. For 
	atoms below $a_{k}$ we have $\delta_{0}=1$ and $\varepsilon_{j}=-1$ for 
	all $j<k$. 
	
	By $Q_{k}$ we have $a_{k}\meet\comp s_{k}=0$ so that $a_{k}\le s_{k}$.
	Also for $0\le i<k$ we have 
	$a_{k}\meet\comp t_{i+1}\le t_{0}\meet\bigwedge_{j=1}^{i}\comp 
	s_{j}\meet\comp t_{i+1}=0$ by $R_{k, i}$ so that $a_{k}\le t_{i}$ for 
	all $i$. This implies 
	$a_{k}=a_{k}\meet s_{k}\meet\bigwedge_{i=1}^{k} t_{i}$ is 
	an atom or zero.
	
	If $a_{k}=0$ then so does 
	$a'_{k}=a_{k}\meet s_{k}\meet\bigwedge_{i=1}^{k} t_{i}$. 
	Therefore (in $\mathcal F_{k}$) we have 
	$a_{k}'=\bigwedge_{i=0}^{k-1}\comp s_{i}\meet
	s_{k}\meet\bigwedge_{i=0}^{k} t_{i}$ is in the ideal generated by 
	$S_{k}$. As $a_{k}'$ is an atom in $\mathcal F_{k}$ this means that
	$a_{k}'$ must be below one of the elements of $S_{k}$. 
	\begin{enumerate}[-]
		\item $a_{k}'\le \comp s_{0}$ \quad so it is not below $s_{0}$. 
	
		\item $a_{k}'\le\comp s_{i}$ for $i=0, \dots, k-1$ \quad so that
		$a_{k}'$ is not below any $u_{i}$. 
	
		\item $a_{k}'\le s_{k}$ \quad so that $a_{k}'$ is not below $q_{k}$. 
	
		\item $a_{k}'\le t_{i}$ for $i=0, \dots, k$ \quad so that $a_{k}'$ is 
		not below $r_{i}$. 
	\end{enumerate}
	Hence $a_{k}'$ cannot be in this ideal. 
\end{proof}

The remainder of the analysis is an investigation of the change from 
$\mathcal B_{k-1}$ to $\mathcal B_{k}$. The last lemma is the most 
important change as in $\mathcal B_{k-1}$ we have $a_{k}=0$ by $Q_{k-1}$.

We need also note that $R_{i}$ is independent of $k$ so another
change is that $R_{k-1}$ comes into effect. As $Q_{k}$ implies $R_{k-1}$ 
since $s_{k}\le t_{k}$ this is not noticeable until $\mathcal 
B_{k+1}$,  so we really only have $R_{k-2}$ and $Q_{k}$ to worry 
about.

Further we need to note that the relations $R_{i}$ and $Q_{k}$ 
only affect the interval $[0, t_{0}]$ and do nothing in $[0, \comp 
t_{0}]$ --
this is because they are all of the form $t_{0}\meet\text{something}=0$.

%\newpage
\begin{lem}\label{lem:atomSplit}
	Let $a\in\mathcal B_{k-1}$ be an atom. Then $a$ is split into three 
	atoms in $\mathcal B_{k}$.
\end{lem}
\begin{proof}
	We will work in $\mathcal F_{k}$ as $\mathcal B_{k}$ is a quotient 
	of this algebra. We will also assume that $a$ is an atom in $\mathcal 
	F_{k-1}$ so that
	$$
		a=\bigwedge_{i=0}^{k-1}\varepsilon_{i}s_{i}\meet\bigwedge_{i=0}^{k-1}\delta_{i}t_{i}.
	$$
	Of course we have $\varepsilon_{0}=-1$ as $s_{0}=0$ in 
	$\mathcal B_{k-1}$.
	
	In $\mathcal F_{k}$ $a$ splits into four parts --
	$a_{00}=a\meet(s_{k}\meet t_{k})$, $a_{01}=a\meet(s_{k}\meet\comp 
	t_{k})$, $a_{10}=a\meet(\comp s_{k}\meet t_{k})$ and 
	$a_{11}=a\meet(\comp s_{k}\meet\comp t_{k})$. In $\mathcal B_{k}$ we 
	have $a_{01}=0$ as $s_{k}\le t_{k}$. We need to show that none of 
	the others are made zero in $\mathcal B_{k}$.
	
	If one of them -- call it $a_{pq}$ -- is zero in $\mathcal B_{k}$ then,  in $\mathcal F_{k}$ 
	it must be in the ideal generated by
	$S_{k}$.
	
	As $a_{pq}$ is an atom in $\mathcal F_{k}$ it must be the case that 
	$a_{pq}$ is smaller than something in $S_{k}$. This means that 
	$\delta_{0}=1$ as everything in $S_{k}$ is below $t_{0}$.
	
	Suppose that $a_{pq}\le t_{0}\meet\bigwedge_{j=1}^{i}\comp s_{j}\meet\comp 
	t_{i+1}$ for some $0\le i<k-1$. As $a_{pq}>0$ this means that 
	$\varepsilon_{j}=-1$ for $0\le j\le i$ and $\delta_{i+1}=-1$. But now 
	we have $a\le t_{0}\meet\bigwedge_{j=1}^{i}\comp s_{j}\meet\comp 
	t_{i+1}$ and so $a=0$ in $\mathcal B_{k-1}$ -- contradiction.
	
	Hence $a_{pq}\le t_{0}\meet\bigwedge_{j=1}^{k-1}\comp s_{j}\meet\comp 
	t_{k}$ or $a_{pq}\le t_{0}\meet\bigwedge_{j=1}^{k}\comp s_{j}$. 
	Either of these implies $\varepsilon_{j}=-1$ for $0\le j\le k-1$ and 
	so $a=0$ in $\mathcal B_{k-1}$ -- contradiction.
\end{proof}

Lastly we need to observe that the new atom $a_{k}=t_{0}\meet\bigwedge_{i=1}^{k-1}\comp s_{j}$
is not one of the atoms produced as in the last lemma -- since if $
a=\bigwedge_{i=0}^{k-1}\varepsilon_{i}s_{i}\meet\bigwedge_{i=0}^{k-1}\delta_{i}t_{i}$
is an atom of $\mathcal B_{k-1}$ and $a_{k}\le a$ in $\mathcal B_{k}$ 
then we have $a_{k}\meet a>0$ in $\mathcal F_{k}$ and so $a\le a_{k}$ 
in $\mathcal F_{k}$. But this implies $\delta_{0}=1$ and
$\varepsilon_{j}=-1$ for $0\le j\le k-1$ and 
	so $a=0$ in $\mathcal B_{k-1}$ -- contradiction.
	
The atoms as produced in the way described above fall into natural 
groupings. It helps to understand the counting arguments we give in 
the next section if we know how these groupings come about. 

\begin{defn}
	Let 
	\begin{align*}
		T_{00}&=\brk<\Set{1}, =>\\
		A_{00}&=\Set{1}\\
		T_{11}&=\brk<\Set{t_{0}}, =>\\
		A_{11}&=\Set{t_{0}}\\
		T_{01}&=\brk<\Set{\comp t_{0}, s_{1}, \comp t_{1}, \comp s_{1}\meet 
		t_{1}}, {\brk<s_{1}, \comp t_{0}>, \brk<\comp t_{1}, \comp t_{0}>, 
		\brk<\comp s_{1}\meet t_{1}, \comp t_{0}>}\cup=>\\
		A_{01}&=\Set{\comp t_{0}, s_{1}, \comp t_{1}, \comp s_{1}\meet 
		t_{1}}
	\end{align*}
	\begin{align*}
		T_{ij}&=
		\begin{cases}
			 \brk<\Set{t_{0}\meet\bigwedge_{l=1}^{i-1}\comp s_{l}}, => & i=j \\
			 \biggl\langle\biggl\{t_{0}\meet\bigwedge_{l=1}^{i-1}\comp s_{l}\meet s_{i}, 
			 t_{0}\meet\bigwedge_{l=1}^{i-1}\comp s_{l}\meet s_{i}\meet s_{i+1}, 
			 t_{0}\meet\bigwedge_{l=1}^{i-1}\comp s_{l}\meet s_{i}\meet\comp 
			 t_{i+1}, & \\
			 \qquad t_{0}\meet\bigwedge_{l=1}^{i-1}\comp s_{l}\meet s_{i}\meet(\comp 
			 s_{i+1}\meet t_{i+1})\biggr\}, 
			 \biggl\{
			 \brk<t_{0}\meet\bigwedge_{l=1}^{i-1}\comp s_{l}\meet s_{i}\meet s_{i+1}, 
			 t_{0}\meet\bigwedge_{l=1}^{i-1}\comp s_{l}\meet s_{i}>, & \\
			 \qquad\qquad
			 \brk<t_{0}\meet\bigwedge_{l=1}^{i-1}\comp s_{l}\meet s_{i}\meet\comp 
			 t_{i+1},t_{0}\meet\bigwedge_{l=1}^{i-1}\comp s_{l}\meet 
			 s_{i}>, &\\
			 \qquad\qquad
			 \brk<t_{0}\meet\bigwedge_{l=1}^{i-1}\comp s_{l}\meet s_{i}\meet(\comp 
			 s_{i+1}\meet t_{i+1}), t_{0}\meet\bigwedge_{l=1}^{i-1}\comp s_{l}\meet 
			 s_{i}>\biggr\}\cup=\biggr\rangle & j=i+1 \\
			 \biggl\langle T_{i, j-1}\cup\bigcup_{\alpha\in A_{i, j-1}}\Set{ 
			 \alpha\meet s_{j}, 
			 \alpha\meet\comp 
			 t_{j}, 
			 \alpha\meet(\comp 
			 s_{j}\meet t_{j})}, & \\
			 \qquad \le_{i, j-1}\cup\bigcup_{\alpha\in A_{i, j-1}}
			 \Set{
			 \brk<\alpha\meet s_{j}, 
			 \beta>, 
			 \brk<\alpha\meet\comp 
			 t_{j},\beta>,  
			 \brk<\alpha\meet(\comp 
			 s_{j}\meet t_{j}), \delta> | \alpha\le_{i, 
			 j-1}\beta}\cup=\biggr\rangle & j>i+1 
		\end{cases}\\ 
		A_{ij}&=
		\begin{cases}
			 \Set{t_{0}\meet\bigwedge_{l=1}^{i-1}\comp s_{l}} & j=i\\
			 \bigl\{t_{0}\meet\bigwedge_{l=1}^{i-1}\comp s_{l}\meet s_{i}, 
			 t_{0}\meet\bigwedge_{l=1}^{i-1}\comp s_{l}\meet s_{i}\meet s_{i+1}, 
			 t_{0}\meet\bigwedge_{l=1}^{i-1}\comp s_{l}\meet s_{i}\meet\comp 
			 t_{i+1}, & \\
			 \qquad\qquad
			 t_{0}\meet\bigwedge_{l=1}^{i-1}\comp s_{l}\meet s_{i}\meet(\comp 
			 s_{i+1}\meet t_{i+1})\bigr\} & j=i+1\\
			 T_{i, j-1}\cup\bigcup_{\alpha\in A_{i, j-1}}\Set{ 
			 \alpha\meet s_{j}, 
			 \alpha\meet\comp 
			 t_{j}, 
			 \alpha\meet(\comp 
			 s_{j}\meet t_{j})} & j>i+1 
		\end{cases}
	\end{align*}
\end{defn}

A few tree diagrams will help us see what this definition is really about.
\begin{center}
	\begin{longtable}{>{$}l<{$}!{:}p{10cm}}
		T_{00} & $\bullet\ 1$\\
		T_{11} & $\bullet\ t_{0}$ \\
		T_{01} & 
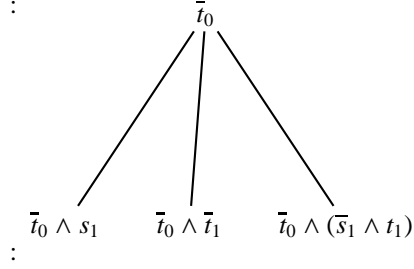
				\begin{minipage}[t]{10cm}
					\pstree[nodesep=2pt, levelsep=80pt]{\TR{$\comp t_{0}$}}{%
					\TR{$\comp t_{0}\meet s_{1}$}
					\TR{$\comp t_{0}\meet\comp t_{1}$}
					\TR{$\comp t_{0}\meet(\comp s_{1}\meet t_{1})$}}
				\end{minipage}\\
		 \vspace{3mm} & \\
		T_{22} & $\bullet\ t_{0}\meet\comp s_{1}$\\
		 \vspace{3mm} & \\
		T_{12} & \begin{minipage}[t]{10cm}
					\pstree[nodesep=2pt, levelsep=20pt]{\TR{$t_0\meet s_1$}}{%
					\TR{$t_0\meet s_1\meet s_{2}$}
					\TR{$t_0\meet s_1\meet\comp t_{2}$}
					\TR{$t_0\meet s_1\meet(\comp s_{2}\meet t_{2})$}}
				\end{minipage}\\
		 \vspace{5mm} & \\
		T_{02} & \begin{minipage}[t]{10cm}
					\pstree[nodesep=2pt, levelsep=80pt, treemode=R,  tnpos=r]{\Item{$\comp t_{0}$}}{%
					\pstree{\Item{$\comp t_{0}\meet s_{1}$}}{%
					\Item{$\comp t_{0}\meet s_{1}\meet s_{2}$}
					\Item{$\comp t_{0}\meet s_{1}\meet \comp t_{2}$}
					\Item{$\comp t_{0}\meet s_{1}\meet (\comp s_{2}\meet t_{2})$}}
					\pstree{\Item{$\comp t_{0}\meet\comp t_{1}$}}{%
					\Item{$\comp t_{0}\meet\comp t_{1}\meet s_{2}$}
					\Item{$\comp t_{0}\meet\comp t_{1}\meet \comp t_{2}$}
					\Item{$\comp t_{0}\meet\comp t_{1}\meet (\comp s_{2}\meet t_{2})$}}
					\pstree{\Item{$\comp t_{0}\meet(\comp s_{1}\meet t_{1})$}}{%
					\Item{$\comp t_{0}\meet(\comp s_{1}\meet t_{1})\meet s_{2}$}
					\Item{$\comp t_{0}\meet(\comp s_{1}\meet t_{1})\meet \comp t_{2}$}
					\Item{$\comp t_{0}\meet(\comp s_{1}\meet t_{1})\meet (\comp s_{2}\meet t_{2})$}}
					}
				\end{minipage}
	\end{longtable}
\end{center}
The reader is invited to produce the next layer of trees.

We note that $\bigcup_{j=0}^{k}A_{jk}$ is the set of all atoms of 
$\mathcal B_{k}$. 

\section{Counting Atoms and Other Things}
Now we want to count just how many atoms there are and in what 
locations they may be found. This leads to a calculation of the size 
of $\mathcal L(X)$.

\begin{defn}
	Let 
	$$
	\alpha_{k}(t)=\text{the number of atoms below }t\text{ in }\mathcal B_{k}.
	$$
\end{defn}

\begin{lem}
	$$
	\alpha_{k}(1)=\frac{1}{2}(3^{k+1}-1).
	$$
\end{lem}
\begin{proof}
	$\mathcal B_{0}=\mathbf 2$ has one atom.
	
	From the lemmas \ref{lem:newAtom} and \ref{lem:atomSplit} we have $\alpha_{k}(1)=3\alpha_{k-1}(1)+1$ 
	from which we have the desired formula.
\end{proof}

\begin{lem}
	$$
	\alpha_{k}((\comp s_{j_{1}}\meet t_{j_{1}})\meet \dots\meet (\comp 
	s_{j_{n}}\meet t_{j_{n}}))=\frac12(3^{k+1-n}-1)
	$$
	if all the $\brk<s_{j_{i}}, t_{j_{i}}>$ are distinct.
\end{lem}
\begin{proof}
	The proof is the same as above -- the first case where the formula 
	makes sense is $\mathcal B_{n-1}$ and in this case 
	$(\comp s_{j_{1}}\meet t_{j_{1}})\meet \dots\meet (\comp 
		s_{j_{n}}\meet t_{j_{n}})\le t_{0}\meet\bigwedge_{i=0}^{n-1}\comp 
		s_{i} =0$
	and so there are no atoms below it,  and $0=\frac12(3^{(n-1)+1-n}-1)$.
	
	Let 
	$$
	t=(\comp s_{j_{1}}\meet t_{j_{1}})\meet \dots\meet (\comp 
	s_{j_{n}}\meet t_{j_{n}}).
	$$
	There are two cases in general -- 
	\begin{description}
		\item[No $j_{i}=k$] 
			Going from $\mathcal B_{k-1}$ to $\mathcal B_{k}$ the new atom
			is below $t$ and all other atoms split in three so we have 
			$\alpha_{k}(t)=3\alpha_{k-1}(t)+1$ which gives the desired formula.
		\item[Some $j_{i}=k$] 
		Without loss of generality $i=n$. 
			This case is different as the new atom $a_{k}$ in $\mathcal B_{k}$
			is not below $\comp s_{k}\meet t_{k}$. But then every atom in
			$\mathcal B_{k-1}$ splits into three, one of which is below
			$\comp s_{k}\meet t_{k}$. Therefore $\alpha_{k}((\comp s_{j_{1}}\meet t_{j_{1}})\meet \dots\meet (\comp 
	s_{j_{n-1}}\meet t_{j_{n-1}})\meet(\comp s_{k}\meet t_{k}))$ is equal to
	$\alpha_{k-1}((\comp s_{j_{1}}\meet t_{j_{1}})\meet \dots\meet (\comp 
		s_{j_{n-1}}\meet t_{j_{n-1}}))$ and so
		equals $\frac12(3^{(k-1)+1-n}-1)=\frac12(3^{k+1-(n+1)}-1)$.
	\end{description}
\end{proof}

As we are really interested in intervals we need the following 
observation
\begin{lem}
	Let $\mathcal B$ be a finite Boolean algebra, and $\mathbf x\in\rsf 
	I(\mathcal B)$. 
	Then the number of atoms in $[\mathbf x, \mathbf 1]$ is
	equal to the number of atoms in $B$ less the number of atoms below
	$\comp x_{0}\meet x_{1}$.
\end{lem}
\begin{proof}
	Let $\mathbf x=[x_{0}, x_{1}]$. We note that if
	$\mathcal A_{1}$ and $\mathcal A_{2}$ are two finite Boolean 
	algebras then the number of atoms in
	$\mathcal A_{1}\times\mathcal A_{2}$ equals the number of atoms in 
	$\mathcal A_{1}$ plus the number of atoms in $\mathcal A_{2}$.
	
	Since we have 
	\begin{align*}
		[\mathbf x, \mathbf 1]&\simeq [0, x_{0}]\times[x_{1}, 1]\\
		&\simeq [0, x_{0}]\times[0, \comp x_{1}]\\
		\intertext{and }
		\mathcal B&\simeq [0, x_{0}]\times[0, \comp x_{1}]\times[0, 
		\comp x_{0}\meet x_{1}]
	\end{align*}
	the result is immediate.
\end{proof}

Now we want to compute the size of $\mathcal L(X)$. As this is a 
union of interval algebras we will use an inclusion-exclusion 
calculation to find its size. We recall that in general for cubic 
algebras that
$$
\mathcal L_{a}\cap\mathcal L_{b}=\mathcal L_{a\join\Delta(a\join b, b)}
$$
is another interval algebra -- \cite{BO:eq} theorem 4.6. 
Using the last lemma it is relatively 
easy to compute the size of these intervals.
\begin{defn}
	Let $\mathbf x$ be an interval in $\mathcal B_{k}$.
	Let 
	\begin{align*}
		\alpha^{*}_{k}(\mathbf x)&=\text{the number of atoms below }\comp 
		x_{0}\meet x_{1};\\
		\alpha^{I}_{k}(\mathbf x)&=\text{the number of atoms in }[\mathbf x, 
		\mathbf 1].
	\end{align*}
\end{defn}

\begin{lem}
	Let $\mathbf x=[x_{0}, x_{1}]$ be an interval in a finite Boolean 
	algebra $\mathcal B$. Let $n$ be the number of atoms in $[\mathbf x, \mathbf 1]$. 
	Then
	$$
	\card{\rsf I(\mathcal B)_{\mathbf x}}=3^{n}.
	$$
\end{lem}
\begin{proof}
	This is just a special case of the fact that if $\card{\mathcal 
	B}=2^{n}$ then $\card{\rsf I(\mathcal B)}=3^{n}$ as 
	$\rsf I(\mathcal B)_{\mathbf x}\simeq\rsf I([\mathbf x, \mathbf 1])$.
\end{proof}

Next the natural points of intersection.
\begin{comment}
\begin{defn}
	\begin{align*}
		\eta_{0}&=[s_{0}, t_{0}]=[0, t_{0}]\\
		\eta_{i+1}&=\eta_{i}\join\Delta(\eta_{i}\join[s_{i+1}, t_{i+1}], [s_{i+1}, t_{i+1}]).
	\end{align*}
\end{defn}

More generally we need 
\end{comment}
\begin{defn}
	Let $I\subseteq\Set{0, \dots, k}$. Let $i_{0}<\dots<i_{l}$ be an 
	increasing enumeration of $I$. Then
	\begin{align*}
		\eta^{I}_{0}&=[s_{i_{0}}, t_{i_{0}}]\\
		\eta^{I}_{j+1}&=\eta^{I}_{j}\join\Delta(\eta^{I}_{j}\join[s_{i_{j+1}}, 
		t_{i_{j+1}}], [s_{i_{j+1}}, 
		t_{i_{j+1}}])&&\text{ if }j<l\\
		\eta^{I}&=\eta^{I}_{l}.
	\end{align*}
\end{defn}

\begin{lem}
	Let $J\subseteq\Set{0, \dots, k}$.
	For all $i$ such that $0\le i<\card J$
	$$
		\eta^{J}_{i}=[s_{j_{0}}\meet\bigwedge_{p=1}^{i}(s_{j_{p}}\join\comp 
		t_{j_{p}}), 
		t_{j_{0}}\join\bigvee_{p=1}^{i}(\comp s_{j_{p}}\meet t_{j_{p}})].
	$$
\end{lem}
\begin{proof}
	The proof is by induction on $i$ -- it is clearly true for $i=0$.
	The superscript $J$ will be suppressed.
	Let 
	\begin{align*}
		a_{i}&=s_{j_{0}}\meet\bigwedge_{p=1}^{i}(s_{j_{p}}\join\comp 
		t_{j_{p}})\\
		b_{i}&=t_{j_{0}}\join\bigvee_{p=1}^{i}(\comp s_{j_{p}}\meet 
		t_{j_{p}})\\
		\eta_{i+1}&=\eta_{i}\join\Delta(\eta_{i}\join[s_{j_{i+1}}, t_{j_{i+1}}], 
		[s_{j_{i+1}}, t_{j_{i+1}}])\\
		&=[a_{i}, b_{i}]\join\Delta([a_{i}\meet s_{j_{i+1}}, b_{i}\join 
		t_{j_{i+1}}], [s_{j_{i+1}}, t_{j_{i+1}}])\\
		&=[a_{i}, b_{i}]\join[(a_{i}\meet s_{j_{i+1}})\join
		(b_{i}\meet\comp t_{j_{i+1}}), (b_{i}\join 
		t_{j_{i+1}})\meet(a_{i}\join\comp s_{j_{i+1}})]\\
		&=[a_{i}\meet(s_{j_{i+1}}\join\comp t_{j_{i+1}}), b_{i}\join(\comp 
		s_{j_{i+1}}\meet t_{j_{i+1}})]\\
		&=[a_{i+1}, b_{i+1}].
	\end{align*}
\end{proof}

The next step is to compute the number of atoms in $\eta^{J}_{i}$ i.e. 
below $\comp a_{i}\meet b_{i}$. 
\begin{lem}
	$$
	\alpha^{I}_{k}(\eta^{J}_{i})=3^{k}\left(\frac{2}{3}\right)^{i}.
	$$
\end{lem}
\begin{proof}
	This is an inclusion-exclusion argument -- let 
	$S_{k, i, J}=\Set{\comp s_{j_{p}}\meet t_{j_{p}} | 0\le p\le i}$ for all 
	$J\subseteq\Set{0, \dots, k}$ and $i<\card J$. Again the superscript $J$ is omitted. 
	
	Then inclusion-exclusion gives us
	\begin{align*}
		\alpha^{*}(\eta_{i})&=\sum_{j=1}^{i+1}(-1)^{j+1}\left(\sum_{%
		\substack{A\subseteq S_{k, i, J}\\
		\card A=j}}%
		\alpha_{k}(\bigwedge A)\right)\\
		&=\sum_{j=1}^{i+1}\binom{i+1}j (-1)^{j+1}\frac12(3^{k+1-j}-1)\\
		&=\frac12(3^{k+1}-1)-\frac12\sum_{j=0}^{i+1}\binom{i+1}j 
		(-1)^{j}(3^{k+1-j}-1)\\
		\sum_{j=0}^{i+1}\binom{i+1}j (-1)^{j}(3^{k+1-j}-1)&=
		\sum_{j=0}^{i+1}\binom{i+1}j 
		(-1)^{j}3^{k+1-j}-\sum_{j=0}^{i+1}\binom{i+1}j (-1)^{j}\\
		&=3^{k+1}\sum_{j=0}^{i+1}\binom{i+1}j\left(-\frac{1}{3}\right)^{j}-
		\sum_{j=0}^{i+1}\binom{i+1}j (-1)^{j}\\
		&=3^{k+1}\left(1-\frac13\right)^{i+1}-(1-1)^{i+1}\\
		&=3^{k+1}\left(\frac23\right)^{i+1}.\\
		\intertext{Thus}
		\alpha^{I}_{k}(\eta_{i})&=\alpha_{k}(1)-\alpha^{*}(\eta_{i})\\
		&=\frac12(3^{k+1}-1)-\left[\frac12(3^{k+1}-1)-3^{k}\left(\frac23\right)^{i}\right]\\
		&=3^{k}\left(\frac23\right)^{i}.
	\end{align*}
\end{proof}

\begin{defn}
	Let
	$$
	\Phi(k, l)=3^{k}\left(\frac23\right)^{l}.
	$$
\end{defn}

Now at last we are able to compute the size of $\mathcal 
L(X)=\bigcup_{i=0}^{k}\mathcal L^{(k)}_{[s_{i}, t_{i}]}$.

For ease of reading let $\mathcal M_{i}=\mathcal L^{(k)}_{[s_{i}, t_{i}]}$.
Then inclusion-exclusion gives us that
\begin{align*}
	\card{\mathcal L(X)}&=\sum_{i=1}^{k+1}(-1)^{i+1}\left(\sum_{%
	\substack{A\subseteq\Set{0, \dots, k}\\
	\card A=i}}\card{\bigcap_{j\in A}\mathcal M_{j}}\right)\\
	&=\sum_{i=1}^{k+1}\binom{k+1}i (-1)^{i+1}3^{\Phi(k, i-1)}.
\end{align*}

\section{The other direction}
Now we turn to showing that the algebra we have constructed 
is the free cubic implication algebra. Since we know that the free algebra embeds 
into $\mathcal L(X)$,  it suffices to show that $\mathcal L(X)$ is generated by the intervals 
$I_{i}=[s_{i}, t_{i}]$ for $i=0, \dots, k$ using only cubic 
operations. In fact it suffices only to show that
the elements covering the $I_{i}$ are all cubically generated as 
these are atoms of $[I_{i}, \mathbf 1]$ and so generate $[I_{i}, 
\mathbf 1]$ with only joins. All other elements in $\mathcal L_{I_{i}}$ 
are then $\Delta$-images of two elements of $[I_{i}, \mathbf 1]$ and so
all of $\mathcal L(X)$ is obtained.

Let $\mathcal F_{k}$ now denote the subalgebra of $\mathcal L(X)$ 
generated by $\Set{ I_{i} | 0\le i\le k}$. We are trying to show that 
$\mathcal F_{k}=\mathcal L(X)$. 

We start with the easiest case.
\begin{lem}
	Every possible atom above $I_{0}$ is in $\mathcal F_{k}$.
\end{lem}
\begin{proof}
	The atoms above $I_{0}$ are of the form $[0, t_{0}\join a]$ where 
	$a\le\comp t_{0}$ is an atom of $\mathcal B_{2}$.
	
	To see we get these we note that 
	\begin{align*}
		[0, t_{0}]\join[s_{i}, t_{i}]&=[0, t_{0}\join t_{i}]\\
		\text{and}\qquad
		([0, t_{0}]\join[\comp t_{i}, \comp s_{i}])\to[0, t_{0}]&=
		[0, t_{0}\join\comp s_{i}]\to[0, t_{0}]\\
		&=[0, t_{0}\join s_{i}].
	\end{align*}
	If $a\le \comp t_{0}$ is any atom then 
	$a=\comp 
	t_{0}\meet\bigwedge_{i}\varepsilon_{i}s_{i}\meet\bigwedge_{i}\delta_{i}t_{i}$
	and therefore
	$t_{0}\join a= 
	\bigwedge_{i}(\varepsilon_{i}s_{i}\join t_{0})\meet
	\bigwedge_{i}(\delta_{i}t_{i}\join t_{0})$ and as we have 
	each $[0, t_{0}\join t_{i}]$ and $[0, t_{0}\join s_{i}]$ we get the 
	desired meets and complements and hence $[0, t_{0}\join a]$.
\end{proof}

The rest of the proof consists of carefully showing that all the 
atoms in $[I_{i}, \mathbf 1]$ are obtained. 

\begin{defn}
	Let $a\in\mathcal B_{k}$ be an atom and $[x, y]\in\rsf I(\mathcal 
	B_{k})$. Then
	\begin{enumerate}[(a)]
	\item
		$a$ is \emph{left-associated with $[x, y]$} iff $a\le x$.
	\item
		$a$ is \emph{right-associated with $[x, y]$} iff $a\le\comp y$.
	\item 
		$a$ is \emph{associated with $[x, y]$} iff $a\le x$ or $a\le\comp y$.
	\end{enumerate}
\end{defn}

The idea of association is that an atom of $[[x, y], \mathbf 1]$ is 
either $[x\meet\comp a, y]$ for some $a$ an atom below $x$, 
or $[x, y\join a]$ for some atom $a\le\comp y$. Thus $a$ is 
associated with $[x, y]$ iff $a$ produces an atom in $[[x, y], \mathbf 
1]$.

\begin{lem}\label{lem:assocBelow}
	Let $a$ be associated with $[x, y]\geq[u, v]$. Then $a$ is associated 
	with $[u, v]$.
\end{lem}
\begin{proof}
	$[u, v]\le[x, y]$ iff $x\le u$ and $v\le y$ iff $x\le u$ and $\comp 
	y\le\comp v$. The result is now clear.
\end{proof}

\begin{lem}
	Let $a$ be an atom of $\mathcal B_{k}$ and $[x, y]$,  $[u, v]$ be two 
	intervals. Then
	\begin{enumerate}[(a)]
		\item  $a$ is left-associated with $[x, y]\join[u, v]$ iff
		$a$ is left-associated with both $[x, y]$ and $[u, v]$.
	
		\item  $a$ is right-associated with $[x, y]\join[u, v]$ iff
		$a$ is right-associated with both $[x, y]$ and $[u, v]$.
	\end{enumerate}
\end{lem}
\begin{proof}
	The result is immediate as $a\le x\meet u$ iff $a\le x$ and $a\le u$;
	and $a\le\comp{y\join v}$ iff $a\le \comp y$ and $a\le\comp v$.
\end{proof}

\begin{lem}\label{lem:assocDelta}
		Let $a$ be an atom of $\mathcal B_{k}$ and $[x, y]$ be an 
	interval. Then $a$ is associated with $[x, y]$ iff
		$a$ is associated with $\Delta(\mathbf 1, [x, y])$.
\end{lem}
\begin{proof}
	$\Delta(1, [x, y])=[\comp y, \comp x]$ so that
	$a\le x$ iff $a\le\comp{\comp x}$, and $a\le\comp y$ iff 
	$a\le\comp y$.
\end{proof}

Our intent is to show that all the desired atoms in $[I_{i}, \mathbf 
1]$ are cubically generated from the $I_{j}$. Because we need to come 
back to this point so often we have the following definition.
\begin{defn}
	Let $a$ be an atom of $\mathcal B_{k}$ associated with an interval 
	$[x, y]\in\mathcal L(X)$. Then $a$ is \emph{cubically assigned to $[x, 
	y]$} iff 
	\begin{enumerate}[-]
			\item  $a\le x$ and $[x\meet\comp a, y]$ is 
			cubically generated from the $I_{j}$; or
		
			\item  $a\le \comp y$ and $[x, y\join a]$ is
			cubically generated from the $I_{j}$.
		\end{enumerate}
\end{defn}

Our task is made a little easier by the fact that we only need to 
show a cubic assignment once in order to get enough cubic 
assignments.
\begin{lem}
	Let $a$ be an atom of $\mathcal B_{k}$ associated with $[x, y]$ and 
	$[u, v]$ both in $\mathcal F_{k}$. Suppose that $a$ is cubically 
	assigned to $[x, y]$.% and $[u, v]$. 
	Then $a$ is cubically assigned to $[u, v]$.
\end{lem}
\begin{proof}
	First it is easy to see that $a$ is cubically assigned to $[x, y]$ 
	iff $a$ is cubically assigned to $\Delta(\mathbf 1, [x, y])$.
	
	This means that we need only deal with the case that $a$ is 
	left-associated with both $[x, y]$ and $[u, v]$. 
	
	Since $a$ is left-associated with both intervals it is also 
	left-associated with $[x\meet u, y\join v]$ and so
	$[x\meet u\meet\comp a, y\join v]>[x\meet u, y\join v]$. Then we can 
	form\\
	$[x\meet u, y\join v]\xrightarrow{[x\meet u\meet\comp a, y\join 
	v]}\relax[u, 
	v] =
	[u\meet\comp a, v]$.
\end{proof}

The next step is to show that each atom does get cubically assigned 
to some interval in $\mathcal F_{k}$.

Let $k\in\N$ and $\rsf S(k)$ be the signed set algebra on $\Set{0, 
\dots, k}$.
We define a function 
$R\colon\Set{\text{atoms of }\mathcal B_{k}}\to\rsf S(k)$ by 
induction on $k$:
\begin{description}
	\item[$k=0$]  The only atom is $1$ which is assigned $\brk<
	\emptyset, \Set{0}>$;

	\item[$k=i+1$]  If $a\in\mathcal B_{i}$ is an atom and 
	$R(a)=\brk<R_{0}, R_{1}>$ then
	\begin{align*}
		R(a\meet s_{k})&=\brk<R_{0}\cup\Set{k}, R_{1}>\\
		R(a\meet \comp t_{k})&=\brk<R_{0}, R_{1}\cup\Set{k}>\\
		R(a\meet(\comp s_{k}\meet t_{k}))&=R(a).
	\end{align*}
	
	For the new atom $R(t_{0}\meet\bigwedge_{j=1}^{i}\comp 
	s_{i})=\brk<\Set{k}, \emptyset>$.
\end{description}

Note that the function $R$ is not onto,  but we do not need it to be.

\begin{lem}\label{lem:Rshape}
	If $R(a)=\brk<R_{0}, R_{1}>$ then
	\begin{enumerate}[(a)]
		\item 
	\begin{tabular}[t]{>{$}r<{$}@{=}>{$}l<{$}}
		R_{0}&\Set{j | a\le s_{j}}\\
		R_{1}&\Set{j | a\le \comp t_{j}}
	\end{tabular}
	\vspace{4pt}
	\item If $j=\min(R_{0}\cup R_{1})$ then  
	$$
	a=(t_{0}\meet\bigwedge_{i=0}^{j-1}\comp s_{i})\meet
	\bigwedge_{p\in R_{0}\setminus(j+1)}s_{p}\meet
	\bigwedge_{q\in R_{1}}\comp t_{q}\meet
	\bigwedge_{
	\substack{r>j\\
	r\notin R_{0}\cup R_{1}}}(\comp s_{r}\meet t_{r}).
	$$
	\end{enumerate}
\end{lem}
\begin{proof}
	This is true for all atoms in $\mathcal B_{0}$ as $1=\comp t_{0}$ is 
	the only atom and $R(1)=\brk<\emptyset, \Set{0}>$.
	
	Suppose that both (a) and (b) are true for all atoms in $\mathcal B_{k}$ and let 
	$a$ be an atom of $\mathcal B_{k+1}$. 
	
	If $a=t_{0}\meet\bigwedge_{i=0}^{k}\comp s_{i}$ is the new atom then 
	we have (by the R-rules) that $a\le t_{i}$ for all $0\le i\le k+1$, 
	and $a\le s_{k+1}$ -- by the Q-rule. Clearly also we have 
	$a\le \comp s_{i}$ for all $0\le i\le k$ so that
	\begin{align*}
		\Set{j | a\le s_{j}}&=\Set{k+1}\\
		\Set{j | a\le \comp t_{j}}&=\emptyset
	\end{align*}
	so that $R(a)$ is as asserted.
	
	It is clear that $a$ has the desired form.
	
	If $a=a'\meet s_{k+1}$ and $a'$ is a $\mathcal B_{k}$-atom which we 
	assume is expressed as in (b). We 
	know inductively that $R(a')$ is as asserted and $a\le s_{k+1}$ and 
	$a\le a'$ so that 
	\begin{align*}
		R_{0}&\subseteq\Set{j | a\le s_{j}}\\
		R_{1}&\subseteq\Set{j | a\le \comp t_{j}}.
	\end{align*}
	We also know that if $j\notin R_{0}$ and $j\le k$ then $a'\meet 
	s_{j}=0$ in $\mathcal B_{k}$ and this is preserved in $\mathcal 
	B_{k+1}$. Likewise with $j\notin R_{1}$. Thus we get the desired 
	equalities.
	
	A similar argument works if $a=a'\meet\comp t_{k+1}$.
	
	If $a=a'\meet(\comp s_{k+1}\meet t_{k+1})$ then $a\nleq s_{k+1}$ and
	$a\nleq\comp t_{k+1}$ so the sets do not change.
\end{proof}

\begin{cor}
	Let $a, b$ be atoms in $\mathcal B_{k}$ such that $R(a)=R(b)$. Then
	$a=b$.
\end{cor}
\begin{proof}
	This is clear from part (b) of the lemma.
\end{proof}

Now we turn to looking at getting atoms assigned to intervals in 
$\mathcal L(X)$. 

\begin{defn}
	Let $\brk<A_{0}, A_{1}>\in\rsf S(k)$. We define the interval
	$$
	J(A_{0}, A_{1})=\bigvee_{j\in A_{0}}I_{j}\join\bigvee_{l\in 
	A_{1}}\Delta(\mathbf 1, I_{l}).
	$$
\end{defn}

We note that $J(A_{0}, A_{1})$ is in $\mathcal F_{k}$. 

\begin{lem}
	Let $a$ be an atom of $\mathcal B_{k}$. Then 
	$a$ is left-associated with $J(R(a))$.
\end{lem}
\begin{proof}
	This is immediate from \lemref{lem:Rshape}(a).
	\begin{comment}
		This is by induction on $k$. In some sense this proof follows $a$'s 
		path through the tree that $a$ is on. 
		
		For $k=0$ we have $a=1$ and $J(R(a))= J(\brk<\emptyset, \Set{0}>)=
		\Delta(\mathbf 1, I_{0})= [1, 1]$.  Clearly $1\le 1$ so that
		$a$ is left-associated with this interval.
		
		For $k>0$ we have two cases.
		The new atom $t_{0}\meet\bigwedge_{i=1}^{k-1}\comp s_{i}\le s_{k}$ by 
		$Q_{k}$.
		
		Otherwise let $a$ be our atom and $R(a)=\brk<R_{0}, R_{1}>$. 
		We need to show that $a$ is 
		left-associated with $I_{l}$ for $l\in R_{0}$ and with $\Delta(\mathbf 1, 
			I_{l})$ for $l\in R_{1}$.
		
		Let $R_{0}\cup R_{1}$ be increasingly enumerated by $j_{1}<\dots<j_{p}$.
	
		If $j_{1}=0$ then $j_{1}\in R_{1}$ and let $a_{1}=\comp t_{0}$.
		By the definition of $R$ we have $a\le a_{1}\le\comp t_{0}$ -- the 
		left endpoint of $\Delta(\mathbf 1, I_{0})$.
		
		If $j_{1}>0$ then $j_{1}\in A_{0}$ and $a\le 
		a_{1}=t_{0}\meet\bigwedge_{i=1}^{j_{1}-1}\comp s_{i}$. In $\mathcal 
		B_{j_{1}}$ we have $a_{1}$ below the left endpoint of $I_{j_{1}}$ and 
		at the next stage we meet $a_{1}$ with $s_{j_{1}}$ and we still have 
		$a\le a_{1}\meet s_{j_{1}}$.
		
		Given that we have $a\le a_{i}$ we define 
		$$
		a_{i+1}=
		\begin{cases}
				 a_{i}\meet s_{j_{i+1}} & \text{ if }j_{i+1}\in A_{0} \\
				 a_{i}\meet \comp t_{j_{i+1}} & \text{ if }j_{i+1}\in A_{1} .
			\end{cases}
		$$
		Then we still have $a\le a_{i+1}$ is below the left endpoint of the 
		appropriate interval.
	\end{comment}
\end{proof}

\begin{lem}\label{lem:assocRR}
	An atom $a$ is associated with an interval of the form $J(A_{0}, A_{1})$
	iff $R(a)\le \brk<A_{0}, A_{1}>$ or $R(a)\le\brk<A_{1}, A_{0}>$.
\end{lem}
\begin{proof}
	Note that $\brk<A_{1}, A_{0}>=\Delta(\mathbf 1, \brk<A_{0}, A_{1}>)$.
	
	The left to right direction is clear as $a$ is associated with 
	$J(R(a))$ and $R(a)\le\brk<A_{0}, A_{1}>$ implies
	$J(\brk<A_{0}, A_{1}>)\le J(R(a))$ so we can apply \lemref{lem:assocBelow}.
	The other half follows from \lemref{lem:assocDelta}.
	
	Suppose that $a$ is associated with $J(A_{0}, A_{1})$. We may assume 
	that $a$ is left-associated -- otherwise use $\Delta(\mathbf 1, 
	J(A_{0}, A_{1}))=J(A_{1}, A_{0})$.
	
	Then we have that $a\le s_{j}$ for all $j\in A_{0}$ and $a\le\comp 
	t_{j}$ for all $j\in A_{1}$ -- by the definition of $J(A_{0}, 
	A_{1})$. Hence we have 
	\begin{align*}
		A_{0}&\subseteq\Set{j | a\le s_{j}}=R_{0}\\
		\text{ and }\qquad
		A_{1}&\subseteq\Set{j | a\le\comp t_{j}}=R_{1}.		
	\end{align*}
	Thus we have $R(a)\le\brk<A_{0}, A_{1}>$. 
\end{proof}

\begin{lem}
	Let $\brk<A_{0}, A_{1}>\in\rsf S(k)$ with $\brk<A_{0}, 
	A_{1}><\brk<\emptyset, \emptyset>$. Then there is some atom 
	$a\in\mathcal B_{k}$ such that either $R(a)=\brk<A_{0}, A_{1}>$ or
	$R(a)=\brk<A_{1}, A_{0}>$.
\end{lem}
\begin{proof}
	Let $A_{0}\cup A_{1}$ be enumerated as $j_{1}<j_{2}<\dots<j_{n}$. 
	\begin{description}
			\item[$j_{1}=0$]  Then we will assume that $j_{1}=0\in A_{1}$ -- 
			else switch the order. Then we define 
			$$
			a=\bigwedge_{j\in A_{0}}s_{j}\meet\bigwedge_{j\in A_{1}}\comp 
			t_{j}\meet\bigwedge_{j\notin A_{0}\cup A_{1}}(\comp s_{j}\meet t_{j}).
			$$
		This is a non-zero atom as it is below $\comp t_{0}$ and nothing 
		gets killed here except by the rules $s_{i}\le t_{i}$.
		Clearly also we have 
		\begin{align*}
			\Set{j | a\le s_{j}}&=A_{0}\\
			\text{ and }\qquad
			\Set{j | a\le\comp t_{j}}&=A_{1}	
		\end{align*}
		so that $R(a)=\brk<A_{0}, A_{1}>$.
		
			\item[$j_{1}>0$]  Then we will assume that $j_{1}\in A_{0}$ -- 
			else switch the order. Then we define 
			$$
			a=(t_{0}\meet\bigwedge_{i=0}^{j_{1}-1}\comp s_{i})\meet
	\bigwedge_{p\in R_{0}\setminus(j_{1}+1)}s_{p}\meet
	\bigwedge_{q\in R_{1}}\comp t_{q}\meet
	\bigwedge_{
	\substack{r>j_{1}\\
	r\notin R_{0}\cup R_{1}}}(\comp s_{r}\meet t_{r}).
			$$
			This is nonzero as nothing below $t_{0}\meet\bigwedge_{i=0}^{j_{1}-1}\comp s_{i}$
			gets killed in any $\mathcal B_{j}$ for $j>j_{1}$.
			Clearly also we have 
		\begin{align*}
			\Set{j | a\le s_{j}}&=A_{0}\\
			\text{ and }\qquad
			\Set{j | a\le\comp t_{j}}&=A_{1}	
		\end{align*}
		so that $R(a)=\brk<A_{0}, A_{1}>$.
		\end{description}
\end{proof}

\begin{comment}
Putting these together we have 
\begin{lem}
	$$
	\card{\Set{a | a\text{ is associated with }J(A_{0}, A_{1})}}=
	\card{\Set{\brk<B_{0}, B_{1}> | \brk<B_{0}, B_{1}>\le\brk<A_{0}, A_{1}>}}.
	$$
\end{lem}
\begin{proof}
	We define a function
	$$
	\phi\colon\Set{a | a\text{ is associated with }J(A_{0}, A_{1})}\to
	\Set{\brk<B_{0}, B_{1}> | \brk<B_{0}, B_{1}>\le\brk<A_{0}, A_{1}>}
	$$
	by
	$$
		\phi(a)=
		\begin{cases}
			R(a) & \text{ if }R(a)\le\brk<A_{0}, A_{1}> \\
			\Delta(\mathbf 1, R(a)) & \text{ if }R(a)\le\brk<A_{1}, A_{0}>. 
		\end{cases}
	$$
	Then $\phi$ is one-one since $R(a)=R(b)$ implies $a=b$, and 
	$R(a)$ cannot equal $\Delta(\mathbf 1, R(b))$ as either
	$0$ can be in $R(a)_{0}\cup R(a)_{1}$ only if it is in $R(a)_{1}$ 
	and if $j=\min(R(a)_{0}\cup R(a)_{1})>0$ then $j\in R(a)_{0}$.
	Since $\Delta(\mathbf 1, \bullet)$ reverse this we see that equality 
	is impossible.
	
	$\phi$ is onto by the last lemma. 
\end{proof}
\end{comment}

Now we can prove that the desired assignments exist. 
\begin{lem}
	Let $a$ be any atom. Then $a$ is cubically assigned to $J(R(a))$.
\end{lem}
\begin{proof}
	The proof is by induction on the rank of $R(a)=\brk<R_{0}, R_{1}>$ 
	in $\rsf S(k)$. 
	
	\begin{description}
			\item[Rank $0$]  Then there is only one atom associated with
			$J(R(a))$ and in particular the interval $[J(R(a)), \mathbf 1]$ 
			has only two elements, and it's unique atom is $[0, 1]$. Since 
			this atom has to come from $a$ we see that $a$ is cubically 
			assigned to $J(R(a))$.
		
			\item[Rank $>0$]  Then for every $\brk<B_{0}, B_{1}><J(R(a))$ there is an 
			atom $b$ such that $R(b)$ is equal to either $\brk<B_{0}, B_{1}>$ 
			or $\brk<B_{1}, B_{0}>$ and is therefore associated with 
			and cubically assigned to $J(B_{0}, B_{1})$. 
			By \lemref{lem:assocRR} $a$ is the only other atom associated with
			$J(R(a))$. 
			
			For each atom $b$ associated with $J(R(a))$ let
			$x_{b}\in[J(R(a)), \mathbf 1]$ be the corresponding atom. 
			Then by induction we cubically have $x_{b}$ for all $b\not=a$ and hence
			$$
			x_{a}=\left(\bigvee_{b\not=a}x_{b}\right)\rightarrow J(R(a))
			$$
			is cubically generated. Thus $a$ is cubically assigned to $J(R(a))$.
		\end{description}
\end{proof}

\begin{thm}
	$$
	\mathcal L(X)\simeq \mathcal Fr(X)
	$$
\end{thm}
\begin{proof}
	Since we have $\mathcal Fr(X)$ embeds into $\mathcal L(X)$
	and every atom $a$ in $\mathcal B_{k}$ is cubically assigned to 
	$J(R(a))\in\mathcal L(X)$ we see that every element of
	$\mathcal L(X)$ is generated by the intervals $I_{i}$. 
	Hence the embedding must be onto.
\end{proof}

This completes our description of free cubic implication algebras. One nice 
consequence of this result is that if 
$\mathcal M$ is an MR-algebra and $X\subseteq \mathcal M$ is any set, 
then the cubic subalgebra generated by $X$ is upwards closed in the 
MR-subalgebra generated by $X$. This is true because our result shows 
that it is true for finite free algebras.


\begin{bibdiv}
\begin{biblist}
    \DefineName{cgb}{Bailey, Colin G.}
    \DefineName{jso}{Oliveira,  Joseph S.}

\bib{BO:eq}{article}{
title={An Axiomatization for Cubic Algebras}, 
author={cgb}, 
author={jso}, 
book={
    title={Mathematical Essays in Honor of Gian-Carlo Rota}, 
    editor={Sagan,  Bruce E.}, 
    editor={Stanley, Richard P.}, 
    publisher={Birkha\"user}, 
    date={1998.}, 
}, 
pages={305--334}
}

\bib{BO:UniMR}{article}{
title={A Universal Axiomatization of Metropolis-Rota Implication Algebras}, 
author={cgb}, 
author={jso},
status={in preparation} ,
eprint={arXiv:0902.0157v1 [math.CO]}
}

\bib{MR:cubes}{article}{
author={Metropolis, Nicholas}, 
author={Rota,  Gian-Carlo}, 
title={Combinatorial Structure of the faces 
of the n-Cube}, 
journal={SIAM J.Appl.Math.}, 
volume={35}, 
date={1978}, 
pages={689--694}
}
\end{biblist}
\end{bibdiv}
\end{document}